\title{\bf Locally common graphs}
\author{Endre Cs\'oka, Tam\'as Hubai\\
Alfr\'ed R\'enyi Institute of Mathematics, Budapest, Hungary, and\\
L\'aszl\'o Lov\'asz\\
Institute of Mathematics, E\"otv\"os Lor\'and University, Budapest,
Hungary
}
\date{}

\documentclass{article}
\usepackage{amssymb,amsmath,graphicx,bbm,srcltx}
\usepackage{hyperref}
\usepackage[amsmath,amsthm,thmmarks]{ntheorem}
\usepackage{tikz}
\usepackage{subcaption}

\DeclareMathSymbol{\shortminus}{\mathbin}{AMSa}{"39}
\DeclareCaptionLabelFormat{subfig}{\figurename #1~\arabic{figure}/\alph{subfigure}:}
\makeatletter\def\p@subfigure{\thefigure/}\makeatother

\newtheorem{theorem}{Theorem}[section]
\newtheorem{prop}[theorem]{Proposition}
\newtheorem{lemma}[theorem]{Lemma}
\newtheorem{claim}{Claim}
\newtheorem{corollary}[theorem]{Corollary}

\theorembodyfont{\rmfamily}

\newenvironment{proof*}[1]{\medskip\noindent{\bf Proof of #1.}}{\hfill$\square$\medskip}

\long\def\ignore#1{}

\def\url{}

\usepackage{soul}
\usepackage{xcolor}
\usepackage{tcolorbox}
\setstcolor{red}

\definecolor{red}{rgb}{1,0,0}
\definecolor{redish}{rgb}{0.8,0,0}
\definecolor{green}{rgb}{0,0.5,0}
\definecolor{blue}{rgb}{0,0,1}
\definecolor{lblue}{rgb}{0.5,0,1}
\definecolor{grey}{rgb}{0.5,0.5,0.5}
\definecolor{orange}{rgb}{1, 0.7, 0}

\begin{document}

\def\Pr{{\sf P}}
\def\E{{\sf E}}
\def\Var{{\sf Var}}
\def\eps{\varepsilon}
\def\wt{\widetilde}
\def\wh{\widehat}

\def\AA{\mathcal{A}}\def\BB{\mathcal{B}}\def\CC{\mathcal{C}}
\def\DD{\mathcal{D}}\def\EE{\mathcal{E}}\def\FF{\mathcal{F}}
\def\GG{\mathcal{G}}\def\HH{\mathcal{H}}\def\II{\mathcal{I}}
\def\JJ{\mathcal{J}}\def\KK{\mathcal{K}}\def\LL{\mathcal{L}}
\def\MM{\mathcal{M}}\def\NN{\mathcal{N}}\def\OO{\mathcal{O}}
\def\PP{\mathcal{P}}\def\QQ{\mathcal{Q}}\def\RR{\mathcal{R}}
\def\SS{\mathcal{S}}\def\TT{\mathcal{T}}\def\UU{\mathcal{U}}
\def\VV{\mathcal{V}}\def\WW{\mathcal{W}}\def\XX{\mathcal{X}}
\def\YY{\mathcal{Y}}\def\ZZ{\mathcal{Z}}

\def\Ab{\mathbf{A}}\def\Bb{\mathbf{B}}\def\Cb{\mathbf{C}}
\def\Db{\mathbf{D}}\def\Eb{\mathbf{E}}\def\Fb{\mathbf{F}}
\def\Gb{\mathbf{G}}\def\Hb{\mathbf{H}}\def\Ib{\mathbf{I}}
\def\Jb{\mathbf{J}}\def\Kb{\mathbf{K}}\def\Lb{\mathbf{L}}
\def\Mb{\mathbf{M}}\def\Nb{\mathbf{N}}\def\Ob{\mathbf{O}}
\def\Pb{\mathbf{P}}\def\Qb{\mathbf{Q}}\def\Rb{\mathbf{R}}
\def\Sb{\mathbf{S}}\def\Tb{\mathbf{T}}\def\Ub{\mathbf{U}}
\def\Vb{\mathbf{V}}\def\Wb{\mathbf{W}}\def\Xb{\mathbf{X}}
\def\Yb{\mathbf{Y}}\def\Zb{\mathbf{Z}}

\def\ab{\mathbf{a}}\def\bb{\mathbf{b}}\def\cb{\mathbf{c}}
\def\db{\mathbf{d}}\def\eb{\mathbf{e}}\def\fb{\mathbf{f}}
\def\gb{\mathbf{g}}\def\hb{\mathbf{h}}\def\ib{\mathbf{i}}
\def\jb{\mathbf{j}}\def\kb{\mathbf{k}}\def\lb{\mathbf{l}}
\def\mb{\mathbf{m}}\def\nb{\mathbf{n}}\def\ob{\mathbf{o}}
\def\pb{\mathbf{p}}\def\qb{\mathbf{q}}\def\rb{\mathbf{r}}
\def\sb{\mathbf{s}}\def\tb{\mathbf{t}}\def\ub{\mathbf{u}}
\def\vb{\mathbf{v}}\def\wb{\mathbf{w}}\def\xb{\mathbf{x}}
\def\yb{\mathbf{y}}\def\zb{\mathbf{z}}

\def\Abb{\mathbb{A}}\def\Bbb{\mathbb{B}}\def\Cbb{\mathbb{C}}
\def\Dbb{\mathbb{D}}\def\Ebb{\mathbb{E}}\def\Fbb{\mathbb{F}}
\def\Gbb{\mathbb{G}}\def\Hbb{\mathbb{H}}\def\Ibb{\mathbb{I}}
\def\Jbb{\mathbb{J}}\def\Kbb{\mathbb{K}}\def\Lbb{\mathbb{L}}
\def\Mbb{\mathbb{M}}\def\Nbb{\mathbb{N}}\def\Obb{\mathbb{O}}
\def\Pbb{\mathbb{P}}\def\Qbb{\mathbb{Q}}\def\Rbb{\mathbb{R}}
\def\Sbb{\mathbb{S}}\def\Tbb{\mathbb{T}}\def\Ubb{\mathbb{U}}
\def\Vbb{\mathbb{V}}\def\Wbb{\mathbb{W}}\def\Xbb{\mathbb{X}}
\def\Ybb{\mathbb{Y}}\def\Zbb{\mathbb{Z}}

\def\R{{\mathbb R}}
\def\Q{{\mathbb Q}}
\def\Z{{\mathbb Z}}
\def\N{{\mathbb N}}
\def\C{{\mathbb C}}
\def\U{{\mathbb U}}
\def\Ge{{\mathbb G}}
\def\Ha{{\mathbb H}}

\def\lunl{[\hskip-1pt[}
\def\runl{]\hskip-1pt]}
\def\one{{\mathbbm1}}

\def\sub{\text{\sf sub}}
\def\inj{\text{\sf inj}}
\def\hom{\text{\sf hom}}
\def\h{\text{\sf h}}

\def\proofend{\hfill$\square$}

\maketitle

\tableofcontents\bigskip

\begin{abstract}
Goodman proved that the sum of the number of triangles in a graph on $n$ nodes
and its complement is at least $n^3/24$; in other words, this sum is minimized,
asymptotically, by a random graph with edge density $1/2$. Erd\H{o}s
conjectured that a similar inequality will hold for $K_4$ in place of $K_3$,
but this was disproved by Thomason. But an analogous statement does hold for
some other graphs, which are called {\it common graphs}. A characterization of
common graphs seems, however, out of reach.

Franek and R\"odl proved that $K_4$ is common in a weaker, local sense. Using
the language of graph limits, we study two versions of locally common graphs.
We sharpen a result of Jagger, \v{S}tov\'{\i}\v{c}ek and Thomason by showing
that no graph containing $K_4$ can be locally common, but prove that all such
graphs are weakly locally common. We also show that not all connected graphs are
weakly locally common.
\end{abstract}

\section{Introduction}

Let $\inj(F,G)$ denote the number of embeddings of the graph $F$ in the graph
$G$. The following inequality was proved by Goodman \cite{Good}:
\begin{equation}\label{EQ:GOOD-COM}
\inj(K_3,G)+\inj(K_3,\overline{G})\ge \frac14|V(G)|^3,
\end{equation}
where equality holds asymptotically if $G$ is a random graph with
edge density $1/2$. Erd\H{o}s conjectured that a similar inequality
will hold for $K_4$ in place of $K_3$, but this was disproved by
Thomason \cite{Tho1} (see also Thomason \cite{Tho2} for a more
``conceptual'' proof). More generally, one can ask which graphs $F$
satisfy
\begin{equation}\label{EQ:COMMON-DEF}
\inj(F,G)+\inj(F,\overline{G})\ge
\bigl(1+o(1)\bigr)2^{1-|E(F)|}|V(G)|^{|V(F)|}
\end{equation}
for every graph $G$, where the $o(1)$ refers to $|V(G)|\to\infty$. Such graphs
$F$ are called {\it common graphs}. So the triangle is common, but $K_4$ is
not. (Throughout the paper, we are going to assume that the graphs are simple
and, unless stressed otherwise, have no isolated nodes.)

Many classes of bipartite graphs are common, and it is conjectured that they
all are. Among non-bipartite graphs, very few are known to be common. Franek
and R\"odl \cite{FraRo} proved that deleting an edge from $K_4$ we get a common
graph. More recently Hatami, Hladky, Kr\'al, Norine and Razborov \cite{HHKNR2}
proved that the $5$-wheel is common, thus providing the first common graph with
chromatic number $4$. In the opposite direction, Jagger, \v{S}tov\'{\i}\v{c}ek
and Thomason \cite{JST} proved that no graph containing $K_4$ is common.

It will be more convenient to count homomorphisms instead of embeddings or
copies of $F$. Let $\hom(F,G)$ denote the number of homomorphisms from $F$ into
$G$. We are interested in the case when $|V(G)|\to\infty$, when
$\inj(F,G)=\hom(F,G)+O(|V(G)|^{|V(F)|-1}$, and so we could replace $\inj$ by
$\hom$ in the definition of common graphs \eqref{EQ:COMMON-DEF}. It will be
even better to consider the normalized version $t(F,G)=\hom(F,G)/|V(G)|^k$,
which can be interpreted as the probability that a random map $\phi:~V(F)\to
V(G)$ preserves adjacency. With this notation, common graphs are those graphs
$F$ for which
\[
t(F,G)+t(F,\overline{G})\ge \bigl(1+o(1)\bigr)2^{1-|E(F)|}
\]
for simple graphs $G$ with $|V(G)|\to\infty$.

Sidorenko \cite{Sid3} studied various ``convexity'' properties of graphs, one
of which is closely related to common graphs. Let us say that a graph $F$ has
the {\it Sidorenko property}, if for every graph $G$,
\[
t(F,G)\ge t(K_2,G)^{|E(F)|}.
\]
It is easy to see that non-bipartite graphs do not have this property, and
Sidorenko conjectured that all bipartite graphs do. A closely related
conjecture, in a different language, was formulated earlier by Simonovits
\cite{Sim}. For us, the significance of this work is that {\it the Sidorenko
property implies that the graph is common.} So the Sidorenko--Simonovits
conjecture would imply that all bipartite graphs are common. Sidorenko's
conjecture has been proved for several rather broad classes of bipartite graphs
\cite{LiSz,ConLee}; for a description of these classes, we refer to these
publications.

Franek and R\"odl \cite{FraRo} proved that $K_4$ is common in a ``local''
sense: the original conjecture of Erd\H{o}s said that the number of $K_4$'s in
a graph and in its complement is minimized asymptotically by a random graph,
and Franek and R\"odl showed that this is true at least for graphs coming from
a random graph by a small perturbation. A more natural formulation of this
result was given in \cite{HomBook}, using notions of graph limit theory (see
below).

Somewhat surprisingly, it turns out that whether or not a graph is ``locally''
common depends on the topology we consider on graph limits. This leads to (at
least) two different versions of this notion: ``locally common'' and ``weakly
locally common''.

More recently Lov\'asz \cite{Lov2011} proved a ``local'' version of Sidorenko's
conjecture, and characterized those graphs satisfying the weak local Sidorenko
property \cite{HomBook}. If a graph is [locally, weakly locally] Sidorenko,
then it is [locally, weakly locally] common, and so these (partial) results
about the Sidorenko property have implications about common graphs. In
particular, all bipartite graphs are locally common.

The goal of this paper is to show that every graph containing $K_4$ is locally
common in the weakest sense, but not in a stronger sense. We give a rather
general sufficient condition for a graph to be weakly locally common, and show
that not all connected graphs are weakly locally common.

\section{Preliminaries}

\subsection{Graph limits}

We need some definition from the theory of graph limits; see \cite{HomBook} for
more detail. A {\it kernel} is a symmetric bounded measurable function
$W:~[0,1]^2\to\R$. (Instead of $[0,1]$ we could use any other standard
probability space here, and we shall do so if it is more convenient.) A {\it
graphon} is a kernel with values in $[0,1]$. We denote the set of kernels by
$\WW$, the set of graphons by $\WW_0$, and the set of kernels with values in
$[-1,1]$ by $\WW_1$.

The significance of graphons is that they provide limit objects for convergent
graph sequences. We call a sequence $(G_1,G_2,\dots)$ of (finite) simple graphs
{\it convergent}, if the numerical sequence $t(F,G_n)$ is convergent for every
simple graph $F$ \cite{BCLSV0}. It was proved in \cite{LSz} that for every
convergent graph sequence there is graphon $W$ such that
\[
t(F,G_n)\to t(F,W)\qquad(n\to\infty),
\]
where
\begin{equation}\label{EQ:T-DEF}
t(F,W)=\int\limits_{[0,1]^{V(F)}} \prod_{ij\in E(F)} W(x_i,x_j)
\prod_{i\in V}\,dx_i.
\end{equation}
Conversely, every graphon represents the limit of a convergent graph sequence.

These results make it possible to formulate our problems in a
remainder-term-free form. A simple graph $F$ is common if and only if
\begin{equation}\label{EQ:COMMON1}
t(F,W)+t(F,1-W)\ge 2^{1-|E(F)|}=2t\Big(F,\frac12\Big)
\end{equation}
for every graphon $W$ (where $1/2$ means the identically-$1/2$
graphon). We can multiply by $2^{|E(F)|}$, and write $W=(1+U)/2$
(where $U\in\WW_1$) to get the inequality
\[
t(F,1+U)+t(F,1-U) \ge 2.
\]

We call a simple graph $F$ {\it locally common for perturbation $\eps>0$}, if
$t(F,1+U)+t(F,1-U) \ge 2$ for every $U\in\WW_1$ with $\|U\|_\infty\le\eps$. We
say that $F$ is {\it locally common}, if there is an $\eps>0$ such that $F$ is
locally common for perturbation $\eps$.

A related notion is that the graph $F$ is {\it weakly locally
common}\footnote{In \cite{HomBook}, only this version was defined and called
``locally common''.}: this means that for every $U\in\WW_1$ there is an
$\eps_U>0$ such that $t(F,1 + \eps U)+t(F,1 -\eps U) \ge 2$ for all
$0\le\eps\le\eps_U$.

It is clear that every common graph is locally common, and every locally common
graph is weakly locally common.
In the other direction, there are weakly locally common graphs which are not locally common, but it is still open whether there are locally common graphs which are not common.

Bipartite graphs are locally common, but
not known to be common. As cited above, Thomason \cite{Tho1} proved that the
graph $K_4$ is not common, while Franek and R\"odl \cite{FraRo} proved (in a
different language) that $K_4$ is weakly locally common. It will follow from
our results that $K_4$ is not locally common. Jagger, \v{S}tov\'{\i}\v{c}ek and
Thomason \cite{JST} proved that no graph containing $K_4$ as a subgraph is
common. We are going to prove that a graph containing $K_4$ is always weakly
locally common, but never locally common.

Similarly to common graphs, we can define ``local'' and ``weakly local''
versions of other extremal problems. We say that a simple graph $F$ has the
{\it local Sidorenko property for perturbation $\eps$}, if $t(F,1+U) \ge 1$ for
every $U\in\WW_1$ with $\int U=0$ and $\|U\|_\infty\le\eps$. It was proved in
\cite{Lov2011} that every bipartite graph $F$ is locally Sidorenko for
perturbation $\eps = 1/(4|E(F)|)$.

We call a simple graph $F$ {\it weakly locally Sidorenko}, if for every
$U\in\WW_1$ with $\int U=0$ there is an $\eps_U>0$ such that $t(F,1 + \eps U)
\ge 1$ for every $0\le\eps\le\eps_U$. The weak local Sidorenko property is even
easier to treat, as noted in \cite{HomBook}, Section 16.5.3: {\it A simple
graph has the weak local Sidorenko property if and only if it is a forest or
its girth is even.}

These results immediately imply some facts about locally common graphs: every
bipartite graph $F$ is locally common for perturbation $1/(4|E(F)|)$, and every
graph with even girth is weakly locally common. We are going to prove a more
general sufficient condition for being weakly locally common.

\subsection{Subgraph densities}\label{SEC:SUBG}

We call a graph {\it mirror-symmetric}, if it is obtained by the following
construction: we take a graph $G$, select a set $S$ of mutually nonadjacent
nodes in it, and glue together two copies of $G$ along $S$.

The following simple fact has been noted in \cite{CCHLL}:

\begin{lemma}\label{LEM:EVENCYC}
If $F$ is mirror-symmetric, then $t(F,U)\ge0$ for every kernel $U$.
\end{lemma}

(It is conjectured in \cite{CCHLL} that this property characterizes mirror-symmetric graphs.)

\begin{proof}
No matter how we fix the variables in the definition of $t(F,U)$ corresponding
to nodes in $S$, integrating the rest gives a square, which is nonnegative.
\end{proof}

We say that a kernel $U$ is {\it balanced}, if $\int_0^1 U(x,y)\,dy=0$ for
almost all $x\in[0,1]$. Analogously, an edge-weighted graph is {\it balanced},
if for every node $v$, the sum of weights of edges incident with $v$ is $0$.

\begin{lemma}\label{LEM:BAL}
A kernel $U$ is balanced if and only if $t(P_3,U)=0)$. If $U$ is a balanced
kernel, and $F$ has a node of degree $1$, then $t(F,U)=0$.
\end{lemma}

\begin{proof}
Lemma \ref{LEM:EVENCYC} implies that $t(P_3,U) \ge 0$ for every kernel $U$. The
case of equality easily follows from the proof of the inequality.

If $\deg_F(u)=1$, and $v$ is its neighbor, then fixing $x_i$ for $i\not= u$,
integrating with respect to $u$ gives $0$, by the definition of being balanced.
\end{proof}

Let $\sub(H,F)$ denote the number of subgraphs of $F$ without isolated nodes
isomorphic to $H$. The densities in the ``perturbed'' graphons can be expanded:
\begin{equation}\label{EQ:T-EXPAND}
t(F,1+U)=\sum_{F'\subseteq F}t(F',U) = \sum_H \sub(H,F) t(H,U).
\end{equation}
Hence
\begin{equation}\label{EQ:EXPANDU}
t(F,1+U)+t(F,1-U)=2\sum_{H:~|E(H)|~\text{even}} \sub(H,F) t(H,U) = 2+2p(F,U),
\end{equation}
where
\begin{equation}\label{EQ:P-COMMON}
p(F,U)=\sum_{H:~0<|E(H)|~\text{even}} \sub(H,F) t(H,U).
\end{equation}
Using this notation, we get the following rephrasing of the definitions of
different versions of the common property.

\begin{prop}\label{PROP:COMM-RXPAND}
{\rm(a)} A graph $F$ is common if and only if $p(F,U)\ge 0$ for all
$U\in\WW_1$.

\smallskip

{\rm(b)} A graph $F$ is locally common if and only if there is a number
$\eps>0$ such that $p(F,\eps U)\ge 0$ for all $U\in\WW_1$.

\smallskip

{\rm(c)} A graph $F$ is weakly locally common if and only if for every
$U\in\WW_1$ there is a number $\gamma_U>0$ such that $p(F,\eps U)\ge 0$ for all
$0<\eps\le\gamma_U$.
\end{prop}

Defining
\begin{equation}\label{EQ:CRFU}
c_r(F,U)=\sum_{H:~|E(H)|=r} \sub(H,F) t(H,U)\quad(r=0,1,\dots),
\end{equation}
we can express $p(F,\eps U)$ as a polynomial in $\eps$:
\begin{equation}\label{EQ:EXPAND}
p(F,\eps U)=\sum_{r=1}^{\lfloor|E(F)|/2\rfloor} \eps^{2r} c_{2r}(F,U).
\end{equation}
Using this expansion, assertion (c) in Proposition \ref{PROP:COMM-RXPAND} can
be rephrased as follows: {\it A graph $F$ is weakly locally common if and only
if for every $U\in\WW_1$, either $c_2(F,U)=c_4(F,U)=\dots=0$, or the first
nonzero number in the sequence $c_2(F,U),c_4(F,U),\dots$ is positive.}

For a short proof of the result of Franek and R\"odl \cite{FraRo} that the
graph obtained from $K_4$ by deleting an edge is common, using this language,
see \cite{HomBook}, Section 16.5.4.

\section{Locally common graphs}

Our goal is to prove the following strengthening of the result of Jagger,
\v{S}tov\'{\i}\v{c}ek and Thomason \cite{JST}, asserting that graphs containing
$K_4$ are never common.

\begin{theorem}\label{PROP:K4-LCOM}
No graph containing $K_4$ is locally common.
\end{theorem}

\begin{proof}
We start with some general consequences of the expansion formulas in the
previous section. Let us introduce two operations on kernels: for a kernel $U$
and $0<\delta\le 1$, define a kernel $U_\delta\in\WW_1$ by
\[
U_\delta(x,y) =
  \begin{cases}
    U(x/\delta,y/\delta), & \text{if $x,y\le \delta$}, \\
    0, & \text{otherwise}.
  \end{cases}
\]
For a kernel $U$ and positive integer $m$, we define the ``tensor power''
kernel $U^{{\otimes}m}:~[0,1]^m\times[0,1]^m\to[-1,1]$ by
\[
U^{{\otimes}m}\bigl((x_1,\dots,x_m),(y_1,\dots,y_m)\bigr)= U(x_1,y_1)\cdots U(x_m,y_m).
\]
It is straightforward that if $U\in\WW_1$ is balanced, then so are $U_\delta$
and $U^{{\otimes}m}$. Furthermore, $t(F,U_\delta)= \delta^{|V(F)|}\,t(F,U)$ and
$t(F,U^{{\otimes}m})=t(F,U)^m$. (We will use an odd $m$ in this construction,
so that the sign of $t(F,U)$ is preserved.)

Substituting these expressions, we get the expansion
\begin{align}\label{EQ:PFED}
p(F,\eps (U^{{\otimes}m})_\delta)&=\sum_{H:~0<|E(H)|~\text{even}}
\sub(H,F)\, \eps^{|E(H)|}\,\delta^{|V(H)|} t(H,U)^m\nonumber\\
&=\sum_{q=2}^{|V(F)|} \delta^q \sum_{H:~|E(H)|~\text{even}\atop |V(H)|=q}
\sub(H,F) \eps^{|E(H)|} t(H,U)^m.
\end{align}

Suppose that $F$ is locally common for perturbation $\eps$. Then $p(F,\eps
U)\ge 0$ for every kernel $U\in\WW_1$, including every kernel of the form
$(U^{{\otimes}m})_\delta$. The parameter $\eps$ is fixed, but we can play with
the parameters $\delta$ and $m$.

Letting $\delta\to 0$, we get that the first nonzero term in the outer sum must
be positive. There is only one term with $q\le 3$, namely $H=P_3$, and by Lemma
\ref{LEM:BAL}, $t(P_3,U)>0$ unless $U$ is balanced. So let us assume that $U$
is balanced. Then Lemma \ref{LEM:BAL} implies that only those terms are nonzero
where all degrees in $H$ are at least $2$. There are only two such graphs with
$q=4$, namely $H=C_4$ and $H=K_4$. Thus (simplifying by $\delta^4\eps^2$) we
get a necessary condition for being locally common for perturbation $\eps$:
\begin{equation}\label{EQ:CKUBAL}
\sub(C_4,F) t(C_4,U)^m + \eps^2 \sub(K_4,F) t(K_4,U)^m\ge 0.
\end{equation}
for every balanced kernel $U\in \WW_1$. Here $t(C_4,U)>0$, so the condition is
trivially satisfied if $F$ contains no $K_4$. Our goal is to prove the
converse.

Letting $m\to\infty$, this implies that
\begin{equation}\label{CKUBAL2}
t(C_4,U)\ge
  \begin{cases}
    -t(K_4,U), & \text{if $\sub(K_4,F)>0$}, \\
    0, & \text{otherwise}.
  \end{cases}
\end{equation}
This strange conclusion, which is independent of $\eps$ and almost independent
of $F$, says the following: either $t(C_4,U)+t(K_4,U)\ge 0$ for every balanced
$U\in\WW_1$, or no locally common graph contains $K_4$. We show that the second
alternative occurs, by constructing a kernel $U$ violating the first
inequality. The construction is carried out in several steps.

\begin{claim}\label{CLAIM:1}
There exists a looped-simple graph $G_1$ with edgeweights $\pm1$ such that
$t(C_4,G_1)+t(K_4,G_1)=-1/4$.
\end{claim}

Let $G_1$ obtained from $K_4$ by adding a loop with weight $-1$ at every node.
Then $t(C_4,G_1)+t(K_4,G_1)=-1/4$ by direct calculation. Note that
$t(C_4,G_1)\ge 0$, so $t(K_4,G_1)<0$.

\begin{claim}\label{CLAIM:2}
There exists an arbitrarily large simple graph $G_2$ (without loops) with
edgeweights $\pm1$ such that $t(C_4,G_2)+t(K_4,G_2)\le -1/5$.
\end{claim}

Indeed, consider any looped-simple graph $G$ with the properties of Claim
\ref{CLAIM:1}, and take its categorical product $G_2=K_n\times G$, where $K_n$
is a large complete graph (without loops). Then $G_2$ has no loops, and
\begin{align*}
t(C_4,G_2)+t(K_4,G_2)&=t(C_4,K_n)t(C_4,G_1)+t(K_4,K_n)t(K_4,G_1)\\
&\to t(C_4,G_1)+t(K_4,G_1) =-\frac14 \qquad(n\to\infty).
\end{align*}
So $t(C_4,G_2)+t(K_4,G_2)\le -1/5$ if $n$ is large enough.

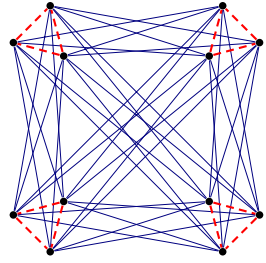
\begin{figure}
\centering
\begin{subfigure}[b]{4.8cm}
\centering
\begin{tikzpicture}
[every node/.style={circle,fill,minimum size=1mm,inner sep=0pt},
poz/.style={blue!50!black,very thin}, neg/.style={red,densely dashed,thick},
el/.style={draw=none,fill=none,midway}]
\clip (-0.8, -0.41) rectangle (3.3, 2.91);
\node (A) at (0.2, 0.2) {};
\node (B) at (2.3, 0.2) {};
\node (C) at (0.2, 2.3) {};
\node (D) at (2.3, 2.3) {};
\draw[poz] (A) -- (B) node[el,anchor=90] {$1$}
    (A) -- (C) node[el,anchor=0] {$1$}
    (A) -- (D) node[el,pos=1/3,anchor=315] {$1$}
    (B) -- (C) node[el,pos=1/3,anchor=225] {$1$}
    (B) -- (D) node[el,anchor=180] {$1$}
    (C) -- (D) node[el,anchor=270] {$1$};
\draw[neg] (A) .. controls +(180:1.3cm) and +(270:1.3cm) .. (A) node[el,pos=1/3,anchor=0] {$\shortminus1$};
\draw[neg] (B) .. controls +(0:1.3cm) and +(270:1.3cm) .. (B) node[el,pos=1/3,anchor=180] {$\shortminus1$};
\draw[neg] (C) .. controls +(180:1.3cm) and +(90:1.3cm) .. (C) node[el,pos=1/3,anchor=0] {$\shortminus1$};
\draw[neg] (D) .. controls +(0:1.3cm) and +(90:1.3cm) .. (D) node[el,pos=1/3,anchor=180] {$\shortminus1$};
\end{tikzpicture}
\caption{$G_1$}
\label{fig:F1a}
\end{subfigure}
\begin{subfigure}[b]{4.8cm}
\centering
\begin{tikzpicture}
[every node/.style={circle,fill,minimum size=1mm,inner sep=0pt},
poz/.style={blue!50!black,ultra thin}, neg/.style={red,densely dashed,thick}]
\path (0, 0) coordinate (A)
    +(45:4mm) node (A1) {}
    +(165:4mm) node (A2) {}
    +(285:4mm) node (A3) {};
\path (2.5, 0) coordinate (B)
    +(135:4mm) node (B1) {}
    +(15:4mm) node (B2) {}
    +(255:4mm) node (B3) {};
\path (0, 2.5) coordinate (C)
    +(315:4mm) node (C1) {}
    +(195:4mm) node (C2) {}
    +(75:4mm) node (C3) {};
\path (2.5, 2.5) coordinate (D)
    +(225:4mm) node (D1) {}
    +(345:4mm) node (D2) {}
    +(105:4mm) node (D3) {};
\foreach \i/\j in {1/2,1/3,2/1,2/3,3/1,3/2}
    \draw[poz] (A\i) -- (B\j) (A\i) -- (C\j) (A\i) -- (D\j)
               (B\i) -- (C\j) (B\i) -- (D\j) (C\i) -- (D\j);
\foreach \i in {A,B,C,D}
    \foreach \j in {1,...,2}
        \pgfmathtruncatemacro{\jn}{\j+1}
        \foreach \k in {\jn,...,3}
            \draw[neg] (\i\j) -- (\i\k);
\end{tikzpicture}
\caption{$G_2 = K_3 \times G_1$}
\label{fig:F1b}
\end{subfigure}
\end{figure}

\begin{figure}
\vspace{3mm}
\centering
\begin{subfigure}[b]{3.3cm}
\centering
\begin{tikzpicture}[xscale=1,yscale=-1]
\newcommand{\psu}{\color{blue!50!black}1}
\newcommand{\ngu}{\color{red}\shortminus1}
\newcommand{\sep}{\hspace{-2mm}}
\path (1, 1) node[inner sep=0pt] {$\left(
\begin{matrix}
\ngu&\sep\psu&\sep\psu&\sep\psu\\
\psu&\sep\ngu&\sep\psu&\sep\psu\\
\psu&\sep\psu&\sep\ngu&\sep\psu\\
\psu&\sep\psu&\sep\psu&\sep\ngu\\
\end{matrix}
\right)
$};
\end{tikzpicture}
\caption{$G_1$}
\label{fig:F2a}
\end{subfigure}
\begin{subfigure}[b]{3.3cm}
\centering
\begin{tikzpicture}[xscale=1/6,yscale=-1/6]
\clip (0, 0) rectangle (12, 12);
\path[fill=blue!50!black,fill opacity=.75]
    (4, 0) -- (6, 0) -- (6, 1) -- (7, 1) -- (7, 2) -- (8, 2) -- (8, 3) -- (9, 3) --
        (9, 4) -- (10, 4) -- (10, 5) -- (11, 5) -- (11, 6) -- (12, 6) -- (12, 8) -- (11, 8) --
        (11, 7) -- (10, 7) -- (10, 6) -- (9, 6) -- (9, 5) -- (8, 5) -- (8, 4) -- (7, 4) --
        (7, 3) -- (6, 3) -- (6, 2) -- (5, 2) -- (5, 1) -- (4, 1) -- cycle
    (7, 0) -- (9, 0) -- (9, 1) -- (10, 1) -- (10, 2) -- (11, 2) -- (11, 3) -- (12, 3) --
        (12, 5) -- (11, 5) -- (11, 4) -- (10, 4) -- (10, 3) -- (9, 3) -- (9, 2) -- (8, 2) --
        (8, 1) -- (7, 1) -- cycle
    (10, 0) -- (12, 0) -- (12, 2) -- (11, 2) -- (11, 1) -- (10, 1) -- cycle
    (3, 1) -- (4, 1) -- (4, 2) -- (5, 2) -- (5, 3) -- (3, 3) -- cycle
    (6, 4) -- (7, 4) -- (7, 5) -- (8, 5) -- (8, 6) -- (6, 6) -- cycle
    (9, 7) -- (10, 7) -- (10, 8) -- (11, 8) -- (11, 9) -- (9, 9) -- cycle
    (1, 3) -- (3, 3) -- (3, 5) -- (2, 5) -- (2, 4) -- (1, 4) -- cycle
    (4, 6) -- (6, 6) -- (6, 8) -- (5, 8) -- (5, 7) -- (4, 7) -- cycle
    (7, 9) -- (9, 9) -- (9, 11) -- (8, 11) -- (8, 10) -- (7, 10) -- cycle
    (0, 4) -- (1, 4) -- (1, 5) -- (2, 5) -- (2, 6) -- (3, 6) -- (3, 7) -- (4, 7) --
        (4, 8) -- (5, 8) -- (5, 9) -- (6, 9) -- (6, 10) -- (7, 10) -- (7, 11) -- (8, 11) --
        (8, 12) -- (6, 12) -- (6, 11) -- (5, 11) -- (5, 10) -- (4, 10) -- (4, 9) -- (3, 9) --
        (3, 8) -- (2, 8) -- (2, 7) -- (1, 7) -- (1, 6) -- (0, 6) -- cycle
    (0, 7) -- (1, 7) -- (1, 8) -- (2, 8) -- (2, 9) -- (3, 9) -- (3, 10) -- (4, 10) --
        (4, 11) -- (5, 11) -- (5, 12) -- (3, 12) -- (3, 11) -- (2, 11) -- (2, 10) -- (1, 10) --
        (1, 9) -- (0, 9) -- cycle
    (0, 10) -- (1, 10) -- (1, 11) -- (2, 11) -- (2, 12) -- (0, 12) -- cycle;
\path[fill=red,fill opacity=.75]
    (1, 0) -- (3, 0) -- (3, 2) -- (2, 2) -- (2, 1) -- (1, 1) -- cycle
    (4, 3) -- (6, 3) -- (6, 5) -- (5, 5) -- (5, 4) -- (4, 4) -- cycle
    (7, 6) -- (9, 6) -- (9, 8) -- (8, 8) -- (8, 7) -- (7, 7) -- cycle
    (10, 9) -- (12, 9) -- (12, 11) -- (11, 11) -- (11, 10) -- (10, 10) -- cycle
    (0, 1) -- (1, 1) -- (1, 2) -- (2, 2) -- (2, 3) -- (0, 3) -- cycle
    (3, 4) -- (4, 4) -- (4, 5) -- (5, 5) -- (5, 6) -- (3, 6) -- cycle
    (6, 7) -- (7, 7) -- (7, 8) -- (8, 8) -- (8, 9) -- (6, 9) -- cycle
    (9, 10) -- (10, 10) -- (10, 11) -- (11, 11) -- (11, 12) -- (9, 12) -- cycle;
\path[fill=lightgray!50,fill opacity=.75]
    (0, 0) -- (1, 0) -- (1, 1) -- (2, 1) -- (2, 2) -- (3, 2) -- (3, 3) -- (4, 3) --
        (4, 4) -- (5, 4) -- (5, 5) -- (6, 5) -- (6, 6) -- (7, 6) -- (7, 7) -- (8, 7) --
        (8, 8) -- (9, 8) -- (9, 9) -- (10, 9) -- (10, 10) -- (11, 10) -- (11, 11) -- (12, 11) --
        (12, 12) -- (11, 12) -- (11, 11) -- (10, 11) -- (10, 10) -- (9, 10) -- (9, 9) -- (8, 9) --
        (8, 8) -- (7, 8) -- (7, 7) -- (6, 7) -- (6, 6) -- (5, 6) -- (5, 5) -- (4, 5) -- (4, 4) --
        (3, 4) -- (3, 3) -- (2, 3) -- (2, 2) -- (1, 2) -- (1, 1) -- (0, 1) -- cycle
    (3, 0) -- (4, 0) -- (4, 1) -- (5, 1) -- (5, 2) -- (6, 2) -- (6, 3) -- (7, 3) --
        (7, 4) -- (8, 4) -- (8, 5) -- (9, 5) -- (9, 6) -- (10, 6) -- (10, 7) -- (11, 7) --
        (11, 8) -- (12, 8) -- (12, 9) -- (11, 9) -- (11, 8) -- (10, 8) -- (10, 7) -- (9, 7) --
        (9, 6) -- (8, 6) -- (8, 5) -- (7, 5) -- (7, 4) -- (6, 4) -- (6, 3) -- (5, 3) --
        (5, 2) -- (4, 2) -- (4, 1) -- (3, 1) -- cycle
    (6, 0) -- (7, 0) -- (7, 1) -- (8, 1) -- (8, 2) -- (9, 2) -- (9, 3) -- (10, 3) --
        (10, 4) -- (11, 4) -- (11, 5) -- (12, 5) -- (12, 6) -- (11, 6) -- (11, 5) -- (10, 5) --
        (10, 4) -- (9, 4) -- (9, 3) -- (8, 3) -- (8, 2) -- (7, 2) -- (7, 1) -- (6, 1) -- cycle
    (9, 0) -- (10, 0) -- (10, 1) -- (11, 1) -- (11, 2) -- (12, 2) -- (12, 3) -- (11, 3) --
        (11, 2) -- (10, 2) -- (10, 1) -- (9, 1) -- cycle
    (0, 3) -- (1, 3) -- (1, 4) -- (2, 4) -- (2, 5) -- (3, 5) -- (3, 6) -- (4, 6) --
        (4, 7) -- (5, 7) -- (5, 8) -- (6, 8) -- (6, 9) -- (7, 9) -- (7, 10) -- (8, 10) --
        (8, 11) -- (9, 11) -- (9, 12) -- (8, 12) -- (8, 11) -- (7, 11) -- (7, 10) -- (6, 10) --
        (6, 9) -- (5, 9) -- (5, 8) -- (4, 8) -- (4, 7) -- (3, 7) -- (3, 6) -- (2, 6) --
        (2, 5) -- (1, 5) -- (1, 4) -- (0, 4) -- cycle
    (0, 6) -- (1, 6) -- (1, 7) -- (2, 7) -- (2, 8) -- (3, 8) -- (3, 9) -- (4, 9) --
        (4, 10) -- (5, 10) -- (5, 11) -- (6, 11) -- (6, 12) -- (5, 12) -- (5, 11) -- (4, 11) --
        (4, 10) -- (3, 10) -- (3, 9) -- (2, 9) -- (2, 8) -- (1, 8) -- (1, 7) -- (0, 7) -- cycle
    (0, 9) -- (1, 9) -- (1, 10) -- (2, 10) -- (2, 11) -- (3, 11) -- (3, 12) -- (2, 12) --
        (2, 11) -- (1, 11) -- (1, 10) -- (0, 10) -- cycle;
\end{tikzpicture}
\caption{$G_2$}
\label{fig:F2b}
\end{subfigure}
\begin{subfigure}[b]{4cm}
\centering
\begin{tikzpicture}[xscale=1/2,yscale=-1/2]
\clip (0, 0) rectangle (4, 4);
\path[fill=blue!50!black,fill opacity=.75]
    (1, 0) -- (4, 0) -- (4, 3) -- (3, 3) -- (3, 2) -- (2, 2) -- (2, 1) -- (1, 1) -- cycle
    (0, 1) -- (1, 1) -- (1, 2) -- (2, 2) -- (2, 3) -- (3, 3) -- (3, 4) -- (0, 4) -- cycle;
\path[fill=red,fill opacity=.75]
    (0, 0) -- (1, 0) -- (1, 1) -- (2, 1) -- (2, 2) -- (3, 2) -- (3, 3) -- (4, 3) --
    (4, 4) -- (3, 4) -- (3, 3) -- (2, 3) -- (2, 2) -- (1, 2) -- (1, 1) -- (0, 1) -- cycle;
\path
    (.5, .5) node {\color{white}$\shortminus1$}
    (1.5, 1.5) node {\color{white}$\shortminus1$}
    (2.5, 2.5) node {\color{white}$\shortminus1$}
    (3.5, 3.5) node {\color{white}$\shortminus1$};
\path[fill opacity=.85]
    (3, 1) node {\color{white}\Large$1$}
    (1, 3) node {\color{white}\Large$1$};
\end{tikzpicture}
\def\clap#1{\hbox to 0pt{\hss#1\hss}}
\def\mathclap{\mathpalette\mathclapinternal}
\def\mathclapinternal#1#2{\clap{$\mathsurround=0pt#1{#2}$}}
\caption{$\smash{\lim\limits_{\mathclap{n \to \infty}}}\; K_n \times G_1$}
\label{fig:F2c}
\end{subfigure}
\caption*{In Figures \ref{fig:F1a} and \ref{fig:F1b}, solid blue lines indicate edges with weight 1, red dashed lines indicate edges with weight $-1$.
Figures \ref{fig:F2a} and \ref{fig:F2b} show their adjacency matrices, blue, grey and red represent 1, 0, $-1$, respectively.
In Figures \ref{fig:F1b} and \ref{fig:F2b}, we used $n=3$, namely, $G_2 = K_3 \times G_1$.
In the language of graph limits, Figure~\ref{fig:F2c} shows the graphon of $K_n \times G_1$ in the limit $n \to \infty$.}
\end{figure}

\begin{claim}\label{CLAIM:3}
There exists a simple graph $G_3$ with balanced edgeweights $\pm1$ such that
$t(C_4,G_3)+t(K_4,G_3) < 0$.
\end{claim}

Let $G_2$ be a graph in Claim \ref{CLAIM:2}, and let $V(G_2)=[r]$. Note that
$r$ can be arbitrarily large.

There is an $r$-uniform $r$-partite hypergraph $H$ with two families of edges
$\{A_1,\dots,A_N\}$ and $\{B_1,\dots,B_N\}$ such that the sets $A_i$ as well as
the sets $B_i$ form a partition of $V(H)$, and $H$ has girth at least $5$. (The
dual hypergraph of an $r$-regular bipartite graph with large girth has these
properties.) Let $V_1,\dots,V_r$ be the partition classes of $H$. We glue a
copy of $G_2$ on every $A_i$ and every $B_i$ (node $u$ of $G_2$ is glued onto
the node of $A_i$ in $V_u$). In the sets $A_i$, we keep the original weighting
of the edges; in the sets $B_i$, we multiply them by $-1$.

It is clear that the weighted graph $G_3$ constructed this way is balanced.
Furthermore, every homomorphism $K_4 \to G_3$ maps $K_4$ into one of the $A_i$
or into one of the $B_i$, and hence $\hom(K_4,G_3) = 2N\,\hom(K_4,G_2)$. This
is not quite true for $C_4$ in place of $K_4$, but the difference is small: it
counts those homomorphisms $C_4\to G_3$ for which two opposite nodes of $C_4$
are mapped onto the same node $v$ of $G_3$, and the other two nodes are mapped
into different copies of $G_2$ containing $v$. Hence
\[
\hom(C_4,G_3)-2N \,\hom(C_4,G_2)\le 2r^3N,
\]
and for $r > 5$,
\begin{align*}
t(C_4,G_3) + t(K_4,G_3)&=\frac{1}{r^4N^4} \Big( \hom(C_4,G_3) + \hom(K_4,G_3) \Big)\\
&\le \frac{1}{r^4N^3} \Big(2\,\hom(C_4,G_2) + 2\,\hom(K_4,G_2) + 2r^3 \Big)\\
&= \frac{2}{N^3} \Big( t(C_4,G_2)+t(K_4,G_2)+ \frac1r \Big) < 0.
\end{align*}
This proves Claim \ref{CLAIM:3}. 
\end{proof}

\begin{figure}
\centering
\begin{tikzpicture}
[every node/.style={circle,fill,minimum size=1mm,inner sep=0pt},
ellA/.style={blue!50!black,very thin}, ellB/.style={red,densely dashed,thick}]
\clip (-16.5mm, -16.5mm) rectangle (16.5mm, 16.5mm);
\node at (4mm, 4mm) {};
\node at (-4mm, 4mm) {};
\node at (-4mm, -4mm) {};
\node at (4mm, -4mm) {};
\node at (11.5mm, 5.5mm) {};
\node at (12.25mm, 12.25mm) {};
\node at (5.5mm, 11.5mm) {};
\node at (-5.5mm, 11.5mm) {};
\node at (-12.25mm, 12.25mm) {};
\node at (-11.5mm, 5.5mm) {};
\node at (-11.5mm, -5.5mm) {};
\node at (-12.25mm, -12.25mm) {};
\node at (-5.5mm, -11.5mm) {};
\node at (5.5mm, -11.5mm) {};
\node at (12.25mm, -12.25mm) {};
\node at (11.5mm, -5.5mm) {};
\draw[ellA] (0mm, 0mm) circle[radius=8mm];
\draw[ellB] (8mm, 8mm) circle[x radius=8mm,y radius=6mm,rotate=45];
\draw[ellB] (-8mm, 8mm) circle[x radius=8mm,y radius=6mm,rotate=135];
\draw[ellB] (-8mm, -8mm) circle[x radius=8mm,y radius=6mm,rotate=45];
\draw[ellB] (8mm, -8mm) circle[x radius=8mm,y radius=6mm,rotate=135];
\draw[ellA] (15mm, 5mm) circle[x radius=6mm,y radius=4mm,rotate=10];
\draw[ellA] (15mm, 15mm) circle[x radius=6mm,y radius=4mm,rotate=45];
\draw[ellA] (5mm, 15mm) circle[x radius=6mm,y radius=4mm,rotate=80];
\draw[ellA] (-5mm, 15mm) circle[x radius=6mm,y radius=4mm,rotate=100];
\draw[ellA] (-15mm, 15mm) circle[x radius=6mm,y radius=4mm,rotate=135];
\draw[ellA] (-15mm, 5mm) circle[x radius=6mm,y radius=4mm,rotate=170];
\draw[ellA] (-15mm, -5mm) circle[x radius=6mm,y radius=4mm,rotate=10];
\draw[ellA] (-15mm, -15mm) circle[x radius=6mm,y radius=4mm,rotate=45];
\draw[ellA] (-5mm, -15mm) circle[x radius=6mm,y radius=4mm,rotate=80];
\draw[ellA] (5mm, -15mm) circle[x radius=6mm,y radius=4mm,rotate=100];
\draw[ellA] (15mm, -15mm) circle[x radius=6mm,y radius=4mm,rotate=135];
\draw[ellA] (15mm, -5mm) circle[x radius=6mm,y radius=4mm,rotate=170];
\end{tikzpicture}
\caption{The local structure of the hypergraph $H$ for $r = 4$. The hyperedges $A_i$ and $B_j$ are shown by blue solid ellipses and red dashed ellipses, respectively.}
\label{fig:F3}
\end{figure}
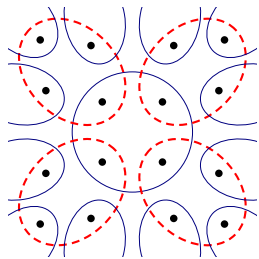

\section{Weakly locally common graphs}

We have seen that every forest and every graph with even girth is weakly
locally common. We prove more in the next Theorem. Let $g_{\rm even}(F)$ denote
the length of the shortest even cycle of $F$, where $g_{\rm even} = \infty$ if
$F$ has no even cycle.

\begin{theorem}\label{THM:WLOC-COM}
If $F$ is not weakly locally common, then $F$ has two (odd) cycles with at most
one node in common, of lengths $g_1$ and $g_2$, such that either $g_1<g_2$ and
$g_1 + g_2 \le g_{\rm even}(F)$ or $g_1 = g_2$ and $g_1 + g_2 < g_{\rm
even}(F)$.
\end{theorem}

In particular, if the length of the shortest even cycle in $F$ is at most twice
of the length of the shortest odd cycle in $F$, or $F$ has no odd cycle, then
$F$ is weakly locally common.

\begin{proof}
Suppose that $F$ is not weakly locally common. Then the sequence $c_2(F,U)$,
$c_4(F,U),\dots$ has a nonzero term, and its first nonzero term, say
$c_{2p}(F,U)$, is negative. We know that $c_2(F,U)\ge0$ by Lemma
\ref{LEM:EVENCYC}, so $p>1$. Hence $c_2(F,U)=0$, which implies that $U$ is
balanced. In this case, $t(H,U)=0$ for every graph $H$ having a node of degree
$1$ by Lemma \ref{LEM:BAL}.

Let $r$ be the smallest positive integer for which $F$ has a subgraph $H$ with
$2r$ edges and $t(H,U) < 0$ (the inequality $c_{2p}(F,U)<0$ implies that such a
subgraph exists and $r\le p$). We know by the above that $2\le r$. Lemma
\ref{LEM:BAL} implies that all degrees in $H$ are at least $2$. We have
$t(F',U) = 0$ for every subgraph $F'$ of $F$ with $|E(F')| < 2r$ and $|E(F')|$
even. Since $U$ is not almost everywhere zero, Lemma \ref{LEM:EVENCYC}(a)
implies that $g_{\rm even} \ge 2r$ and since $H$ is not an even cycle, it
cannot contain an even cycle.

It is a well-known elementary exercise that every block (2-connected component)
of such a graph is an odd cycle. So $H$ contains two odd cycles $C$ and $C'$ of
lengths $g_1$ and $g_2$ intersecting in at most one node. Therefore, $g_1 + g_2
\le |E(H)|=2r \le  g_{\rm even}$.

To complete the proof, we have to exclude the case $g_1=g_2=r$. In this case $H
= C \cup C'$, and $H$ is mirror-symmetric, which implies by Lemma
\ref{LEM:EVENCYC} that $t(H, U) \ge 0$.
\end{proof}

\begin{corollary}\label{COR:C4LOC}
Every graph containing $C_4$ or $C_6$ is weakly locally common.
\end{corollary}

\begin{prop}\label{PROP:NON-LOC-COM}
There exist connected graphs that are not weakly locally common.
\end{prop}

\begin{proof}
Let $F$ consist of a triangle and a pentagon, attached to each other at one
node $u$. (Figure~\ref{fig:F4a}.)
We construct a balanced edge-weighted graph $G$ with edgeweights
$\pm1$ such that $t(F,G) < 0$. (Figure~\ref{fig:F4b}.)
We start with a $4$-star with center node $v$
and endnodes $a,b,c,d$. Let $k$ be a large positive integer. We connect $a$ and
$b$ by an edge; we attach $k$ openly disjoint paths $Q_1,\dots,Q_k$ of length
$3$ and $k+1$ further openly disjoint paths $R_1,\dots,R_{k+1}$ of length $5$ connecting
$c$ and $d$. We weight the following edges with $-1$: the edges $va$ and $vb$;
the middle edge of every path $Q_i$; and every second edge of each path $R_i$,
starting at the end. The remaining edges are weighted with $1$. It is clear
that the weighting is balanced.

\begin{figure}
\centering
\begin{subfigure}[b]{4.5cm}
\centering
\begin{tikzpicture}
[scale=1.25,every node/.style={circle,fill,minimum size=1mm,inner sep=0pt},
edges/.style={blue!50!black,very thin}]
\path (0, 0)
    ++(126:1cm) node (A) {}
    ++(198:1cm) node (B) {}
    ++(270:1cm) node (C) {}
    ++(342:1cm) node (D) {}
    ++(54:1cm) node (E) [label=180:{$u$}]{}
    ++(330:1cm) node (F) {}
    ++(90:1cm) node (G) {};
\draw[edges] (A) -- (B) (B) -- (C) (C) -- (D) (D) -- (E)
             (E) -- (F) (F) -- (G) (G) -- (E) (E) -- (A);
\end{tikzpicture}
\caption{The graph $F$.}
\label{fig:F4a}
\end{subfigure}
\begin{subfigure}[b]{5.8cm}
\centering
\begin{tikzpicture}
[every node/.style={circle,fill,minimum size=1mm,inner sep=0pt},
poz/.style={blue!50!black,very thin}, neg/.style={red,densely dashed,thick}]
\node (a) at (-8mm, 10mm) [label=135:{$a$}]{};
\node (b) at (-8mm, -10mm) [label=225:{$b$}]{};
\node (v) at (0, 0) [label=180:{$v$}]{};
\node (c) at (8mm, 10mm) [label=135:{$c$}]{};
\node (d) at (8mm, -10mm) [label=225:{$d$}]{};
\node (q11) at (6mm, 2.5mm) {};
\node (q12) at (6mm, -2.5mm) {};
\node (q21) at (7.5mm, 2.6mm) {};
\node (q22) at (7.5mm, -2.6mm) {};
\node (q31) at (9mm, 2.7mm) {};
\node (q32) at (9mm, -2.7mm) {};
\node (r11) at (12.5mm, 4.6mm) {};
\node (r12) at (13.5mm, 1.6mm) {};
\node (r13) at (13.5mm, -1.6mm) {};
\node (r14) at (12.5mm, -4.6mm) {};
\node (r21) at (14mm, 4.9mm) {};
\node (r22) at (15mm, 1.8mm) {};
\node (r23) at (15mm, -1.8mm) {};
\node (r24) at (14mm, -4.9mm) {};
\node (r31) at (15.5mm, 5.2mm) {};
\node (r32) at (16.5mm, 2mm) {};
\node (r33) at (16.5mm, -2mm) {};
\node (r34) at (15.5mm, -5.2mm) {};
\node (r41) at (17mm, 5.5mm) {};
\node (r42) at (18mm, 2.2mm) {};
\node (r43) at (18mm, -2.2mm) {};
\node (r44) at (17mm, -5.5mm) {};
\draw[poz] (a) -- (b) (v) -- (c) (v) -- (d)
           (c) -- (q11) (c) -- (q21) (c) -- (q31)
           (q12) -- (d) (q22) -- (d) (q32) -- (d)
           (r11) -- (r12) (r21) -- (r22) (r31) -- (r32) (r41) -- (r42)
           (r13) -- (r14) (r23) -- (r24) (r33) -- (r34) (r43) -- (r44);
\draw[neg] (a) -- (v) (b) -- (v)
           (q11) -- (q12) (q21) -- (q22) (q31) -- (q32)
           (c) -- (r11) (c) -- (r21) (c) -- (r31) (c) -- (r41)
           (r12) -- (r13) (r22) -- (r23) (r32) -- (r33) (r42) -- (r43)
           (r14) -- (d) (r24) -- (d) (r34) -- (d) (r44) -- (d);
\end{tikzpicture}
\caption{The graph $G$ with $k = 3$.}
\label{fig:F4b}
\end{subfigure}
\end{figure}

We claim that
\begin{equation}\label{EQ:NEGT}
t(F,G) < 0.
\end{equation}
The normalization is irrelevant, so it suffices to show that $\hom(F,G) < 0$.
Let $\phi:~V(F)\to V(G)$ be a homomorphism. The triangle in $F$ must be mapped
onto the triangle in $G$. If the pentagon in $F$ is mapped into the subgraph
$G[S]$ induced by $S = \{v,a,b,c,d\}$, then the contribution of $\phi$ is
positive, but the number of these maps is independent of $k$ ($52$, in fact).
If the image of the pentagon contains a node outside $S$, then it must contain
one of the paths $Q_i$, and then $u$ must be mapped onto $v$. The contribution
from such a map is $-1$, and the number of such maps is $4k$. Thus $\hom(F,G) =
52 - 4k$, which is negative if $k>13$. This proves \eqref{EQ:NEGT}.

The condition that $G$ is balanced implies that $t(F',W_G) = 0$ if $F'$ has a
node with degree $1$. The only subgraph of $F$ with an even number of edges and
with all degrees at least $2$ is $F$ itself, and hence $c_2(F,W_G) = c_4(F,W_G)
= c_6(F,W_G) = 0$ but $c_8(F,W_G) = t(F,W_G) < 0$. Thus $F$ is not weakly
locally common.
\end{proof}

\section{Open problems}

In the definition of locally common graphs, we can consider various norms on
the space $\WW$ instead of the $L_\infty$ norm. Can the results above be
extended to other norms? An important candidate is the {\it cut norm}, defined
by
\[
\|W\|_\square = \sup_{S,T\subseteq[0,1]}\left|\int_{S\times T}
W(x,y)\,dx\,dy\right|,
\]
playing an important role in the theory of graph limits. It was proved in
\cite{Lov2011} that every bipartite graph is locally Sidorenko with respect to
the cut norm. Since the cut norm is continuous with respect to every
``reasonable'' norm on $\WW$ (for an exact formulation of this fact see
\cite{HomBook}, Theorem 14.10), it follows that every bipartite graph is
locally Sidorenko in every ``reasonable'' norm on $\WW$.

Similarly to common graphs and Sidorenko graphs, we can define ``local'' and
``weakly local'' versions of other extremal properties, but little is known in
this direction.

Are there any non-common graphs that are locally common in the cut norm or the
$L_\infty$ norm? Is there a graph that is locally common with respect to the
$L_\infty$ norm, but not with respect to the cut norm? Can weakly locally
common graphs be characterized similarly as weakly locally Sidorenko graphs?

\bigskip
\noindent
{\bf Acknowledgement.} 
The research was supported by European Research Council Synergy grant No. 810115.


\begin{thebibliography}{99}

\bibitem{BCLSV0} C.~Borgs, J.~Chayes, L.~Lov\'asz, V.T.~S\'os
    and K.~Vesztergombi: Counting graph homomorphisms, in: {\it
    Topics in Discrete Mathematics} (ed. M.~Klazar, J.~Kratochvil,
    M.~Loebl, J.~Matou\v{s}ek, R.~Thomas, P.~Valtr), Springer (2006),
    315--371.

\bibitem{BCLSV1} C.~Borgs, J.T.~Chayes, L.~Lov\'asz, V.T.~S\'os
    and K.~Vesztergombi: Convergent Graph Sequences I: Subgraph
    frequencies, metric properties, and testing,
    {\it Advances in Math.} {\bf 219} (2008), 1801--1851.


\bibitem{CCHLL} O.A.~Camarena, E.~Cs\'oka, T.~Hubai, G.~Lippner and L.~Lov\'asz:
    Positive graphs, {\it European Journal of Combinatorics} {\bf 52}, Part
    B (2016), 290--301.

\bibitem{ConLee}
    D.~Conlon and J.~Lee: Sidorenko's conjecture for blow-ups,\\
    https://arxiv.org/abs/1809.01259

\bibitem{FraRo} F.~Franek and V.~R\"odl: Ramsey Problem on
    Multiplicities of Complete Subgraphs in Nearly Quasirandom
    Graphs, {\it Graphs and Combinatorics} {\bf8} (1992), 299--308.

\bibitem{Good} A.W.~Goodman: On sets of aquaintences and
    strangers at any party, {\it Amer. Math. Monthly} {\bf 66} (1959)
    778--783.

\bibitem{HHKNR2} H.~Hatami, J.~Hladky, D.~Kral, S.~Norine and
    A.~Razborov: Non-three-colorable common graphs exist,
    \url{http://arxiv.org/abs/1105.0307}

\bibitem{JST} C.~Jagger, P.~\v{S}tov\'{\i}\v{c}ek and
    A.~Thomason: Multiplicities of subgraphs, {\it Combinatorica}
    {\bf 16} (1996), 123--141.

\bibitem{LiSz}
    J.L.X.~Li and B.~Szegedy: On the logarithimic calculus and
    Sidorenko's conjecture,\\
    http://arxiv.org/abs/math/1107.1153

\bibitem{Lov2011} L.~Lov\'asz: Subgraph densities in signed
    graphons and the local Sidorenko conjecture, {\it Electr.
    J.~Combin.} {\bf 18} (2011), P127 (21pp).

\bibitem{HomBook} L.~Lov\'asz: {\it Large networks and graph limits},
    Amer.\ Math.\ Soc., Providence, R.I. (2012).

\bibitem{LSz} L.~Lov\'asz and B.~Szegedy: Limits of dense
    graph sequences, {\it J.~Combin.\ Theory B}
    {\bf 96} (2006), 933--957.

\bibitem{Sid1} A.F.~Sidorenko: Inequalities for functionals
    generated by bipartite graphs (Russian) {\it Diskret. Mat.} {\bf
    3} (1991), 50--65; translation in {\it Discrete Math. Appl.} {\bf
    2} (1992), 489--504.

\bibitem{Sid2} A.F.~Sidorenko: A correlation inequality for
    bipartite graphs, {\it Graphs and Combin.} {\bf 9} (1993), 201--204.

\bibitem{Sid3} A.F.~Sidorenko: Randomness friendly graphs, {\it
    Random Struc.\ Alg.} {\bf 8} (1996), 229--241.

\bibitem{Sim} M.~Simonovits: Extremal graph problems,
    degenerate extremal problems, and supersaturated graphs, in: {\it
    Progress in Graph Theory}, NY Academy Press (1984), 419--437.

\bibitem{Tho1} A.~Thomason: A disproof of a conjecture of
    Erd\H{o}s in Ramsey theory, {\it J.~London Math. Soc.} {\bf39}
    (1898), 246--255.

\bibitem{Tho2} A.~Thomason: Graph Products and Monochromatic
    Multiplicities, {\it Combinatorica} {\bf17} (1997), 125--134.

\end{thebibliography}
\end{document}